\documentclass[]{siamart1116}

\usepackage{amsmath,amsfonts,amssymb,graphicx}
\usepackage{url}
\usepackage{mathrsfs}
\usepackage{xcolor}
\usepackage{times,enumerate}
\usepackage{array}
\ifpdf
  \DeclareGraphicsExtensions{.eps,.pdf,.png,.jpg}
\else
  \DeclareGraphicsExtensions{.eps}
\fi

\numberwithin{theorem}{section}

\newcommand{\TheTitle}{Modeling variational inpainting methods with splines} 
\newcommand{\TheAuthors}{F. Bo\ss{}mann, T. Sauer, N. Sissouno}

\headers{\TheTitle}{\TheAuthors}

\title{{\TheTitle}}

\author{
  Florian Bo\ss{}mann\thanks{Chair of Digital Image Processing, University of Passau, Innstr. 43, 94032 Passau, Germany.}
  \and
  Tomas Sauer\footnotemark[1]
  \and
  Nada Sissouno\footnotemark[1] \thanks{Department of Mathematics, Technical University of Munich, Boltzmannstr. 3, 85748 Garching, Germany.}
}

\usepackage{amsopn}

\newcommand{\R}{\ensuremath\mathbb{R}}
 \newcommand{\Z}{\ensuremath\mathbb{Z}}
 
 \newcommand{\N}{\ensuremath\mathbb{N}}

\newcommand{\argmin}{\operatorname{argmin}}

\newcommand{\ii}{\ensuremath{\mathrm{I}}}
\newcommand{\ik}{\ensuremath{\mathrm{K}}}
\newcommand{\iJ}{\ensuremath{\mathrm{J}}}

\newcommand{\mB}{\ensuremath{\mathbf{B}}}
\newcommand{\mf}{\ensuremath{\mathbf{f}}}

\newcommand{\n}{\ensuremath\pmb{n}}
\newcommand{\m}{\ensuremath\pmb{m}}
\newcommand{\fk}{\ensuremath\pmb{k}}

\newcommand{\lh}{\ensuremath\pmb{h}}
\newcommand{\fx}{\ensuremath\pmb{x}}

\newcommand{\fxi}{\ensuremath\pmb{\xi}}
\newcommand{\ftheta}{\ensuremath\pmb{\theta}}
\newcommand{\TPSs}{\ensuremath\mathscr{S}^{\n}}
\newcommand{\fa}{\ensuremath\pmb{\alpha}}
\newcommand{\fb}{\ensuremath\pmb{\beta}}
\newcommand{\fg}{\ensuremath\pmb{\gamma}}

\newtheorem{Pro}[theorem]{Proposition}

\newtheorem{OP}{Optimization Problem}
\newtheorem{SIM}{Spline Inpainting Model}
\newtheorem{dSIM}{Discrete Spline Inpainting Model}
\newtheorem{remark}[theorem]{Remark}

\ifpdf
\hypersetup{
  pdftitle={\TheTitle},
  pdfauthor={\TheAuthors}
}
\fi

\begin{document}

\maketitle
\begin{abstract}
Mathematical methods of image inpainting involve the discretization of given continuous models. We present a method that avoids the standard pointwise discretization by modeling known variational approaches, in particular total variation (TV), using a finite dimensional spline space. Besides the analysis of the resulting model, we present a numerical implementation based on the alternating method of multipliers. We compare the results numerically with classical TV inpainting and give examples of applications.
\end{abstract}



\section{Introduction}\label{sec:Intro}
In this paper we investigate the modeling of a continuous inpainting
method for digital images. Overviews of
mathematical models/methods in image processing including inpainting
are given in \cite{CS05,AK06}.
Due to the well-known work of Rudin and Osher~\cite{RO94}, and
Rudin, Osher, and Fatemi~\cite{ROF92}, functions of bounded
variation are often considered in the continuous model. For a numerical
solution and its implementation, the models need to be discretized. For
digital images, given by a set of uniformly distributed pixels,
those functions and their derivatives are typically discretized using
difference schemes for the pixel values. In contrast to that,
we use a finite
dimensional function space here, specifically, the space of tensor product
spline functions, see e.g., \cite{dB01,Schu}, and thus avoid
a pointwise discretization.

In general, \emph{inpainting} or \emph{filling in} pursues the goal to
restore parts 
$\Omega$ of an image $\Omega'$, $\Omega\subset\Omega'\subset\R^d$,
where the information has been removed, damaged, or is missing by
using the information from the remaining and well-maintained part.
There
exists a variety of models for this purpose, see
e.g., \cite{BSCB00,BVSO03,EL99} and also \cite{CS05,AK06}.
We will incorporate a variational
approach by, roughly speaking, minimizing some functional
over an extended inpainting area $U(\Omega)\supset\Omega$
subject to side conditions which ensure the reproduction of the
well-maintained part of the image.
Typically, the functional involves the image
function $u$ and some of its derivatives. The side condition, on the
other hand, is chosen over some neighborhood
$B\subseteq\Omega'\setminus\Omega$ of the inpainting area. 
We will model the class of functions for the variational method with
tensor product B-splines together with a focus on \emph{TV
  inpainting}, \cite{CS02}, that is
\begin{equation*}
\int_{U(\Omega)}|\nabla u(\fx)|\, d\fx,
\end{equation*}  
the effective, often used ROF-functional,
see \cite{ROF92}, which results in level curves of minimal length.

We begin by giving some basic properties and notation for the tensor
product splines used in this paper in Section~\ref{sec:Pre_Not}. 
Afterwards, in Section~\ref{sec:Model} we show how to model the
inpainting problem with those splines and analyze its
properties. Numerical results are given in Section~\ref{sec:Ex}. We present
an implementation using the alternating method of multipliers
(ADMM) \cite{CP11}, compare the results of our method to standard TV
inpainting and show some examples of applications.

\section{Preliminaries and notation}\label{sec:Pre_Not}
In the sequel, we will use the following notations
and basic facts about B-splines and their derivatives. Even if images
are usually only considered in 2D, we will present the theory in an
arbitrary number of variables; applications in higher dimensions
would, for example include inpainting in voxel data provided by
computerized tomography.
The tensor product
B-splines for inpainting which we will use in Section~\ref{sec:Model}
are defined over a rectangle
$R=\otimes_{j=1}^d[a_j,b_j]$, $d\in\N$. They are based on 
a {\em tensor product knot grid} which is given as
\begin{equation}\label{eq:knots}
  T
  :=
  \otimes_{j=1}^{d}\{\tau_{j,1},\dots,\tau_{j,m_j+2n_j}\}
\end{equation}
with $\tau_{j,i}<\tau_{j,i+n_j}$ for $n_j\le i \le m_j+n_j$. At the
boundary, we request multiple knots
\begin{equation}\label{eq:boundary_knots}
\tau_{j,1}=\dots=\tau_{j,n_j}=a_j
\quad\textrm{ and }\quad
\tau_{j,m_j+n_j+1}=\dots=\tau_{j,m_j+2n_j}=b_j,
\end{equation}
for $n_j,\, m_j\in\N$, $1\le j\le d$.
We set $\n:=(n_1,\dots, n_d)$ and $\m:=(m_1,\dots,m_d)$ for the
dimensions and degrees in the individual coordinates, respectively.
The {\em grid width}
$\lh:=(h_1,\dots,h_d)$, defined as
$h_j:=\max_k|\tau_{j,k}-\tau_{j,k+1}|$, is known to influence the
approximation properties of the spline space.
A {\em grid cell} is denoted by
$Z_{\fk}:=\otimes_{j=1}^d[\tau_{j,k_j},\tau_{j,k_j+1}]$
for
$\fk\in \ik:=\otimes_{j=1}^d\{1,\dots,m_j+2n_j-1\}$.

The {\em tensor product B-splines} of order $\n$ with respect to the
grid $T$ are denoted by
\[
  B_{\fa}^{\n}(\fx)=\prod_{j=1}^{d}B_{\alpha_j}^{n_j}(x_j), \qquad
  \fa \le \m + \n,
\]
i.e.,
$\fa\in\ii_{\m,\n}:=\otimes_{j=1}^d\{1,\dots,m_j+n_j\}$. They have the
{\em support}
\begin{equation}\label{eq:positivB}
S_{\fa}^{\n}:=\otimes_{j=1}^d S_{\alpha_j}^{n_j}, \qquad
S_{\alpha_j}^{n_j}:=[\tau_{j,\alpha_j},\tau_{j,\alpha_j+n_j}],
\end{equation}
respectively.
The {\em spline space} $\TPSs(T,\Omega)$ of order
$\n$ restricted to a domain $\Omega\subset\R^d$ is spanned by
all B-splines $B_{\fa}^{\n}$ with $S_{\fa}^{\n}\cap \Omega\neq
\emptyset$. The index set of all B-splines relevant for $\Omega$ is
defined as
$\ii_\Omega
:=
\{
\fa\in\ii_{\m,\n}|S_{\fa}^{\n}\cap\Omega\neq\emptyset
\}$
and their number as $\#(\ii_\Omega)$. If $\Omega=R$, then
$\ii_R=\ii_{\m,\n}$ and $\#(\ii_R) = \prod_{j=1}^d m_j+n_j$.\\
A tensor product spline of order $\n$ with respect to $T$ over a domain
$\Omega$ is given by
\[
s(\fx)
=
\sum_{\fa\in \ii_\Omega}f_{\fa}B_{\fa}^{\n}(\fx),\qquad \fx\in\Omega,
\]
with coefficients $f_{\fa}\in\R$. Let
$\mf=\big(f_{\fa}\big)_{\fa\in \ii_\Omega}$ and 
$\mB(\fx)^{T}:=\big(B_{\fa}^{\n}(\fx)\big)_{\fa\in \ii_\Omega}$ denote
the vector of the coefficients and B-splines, respectively, then
the spline $s$ is given in matrix notation by $s(\fx)=\mB(\fx)\mf$.

With respect to inpainting methods an important property of the
spline functions is that their derivatives can be expressed in terms
of the coefficients of the spline function itself, that is 
\begin{equation}\label{eq:derivativeS}
\partial_js(\fx) := 
 \frac{\partial}{\partial x_j}s(\fx)
=\sum_{\fa\in \ii_\Omega}f_{\fa}\,\partial_jB^{\n}_{\fa}(\fx)
\end{equation}
where 
\begin{equation*}\label{eq:derivativeB}
  \partial_jB^{\n}_{\fa}(\fx)
  :=
  \frac{\partial}{\partial x_j} B^{\n}_{\fa}(\fx)
  =
  (n_j-1)
  \left(
  \frac{B^{\n-\varepsilon_j}_{\fa}(\fx)}{|S^{n_j-1}_{\alpha_j}|}
  -
  \frac{B^{\n-\varepsilon_j}_{\fa+\varepsilon_j}(\fx)}
  	{|S^{n_j-1}_{\alpha_j+1}|}
  \right)
\end{equation*}
denotes the derivative of the B-splines. Here, $\varepsilon_j$ is the $j$-th unit vector.

The $\ell^{p}$-norm of vectors will be denoted by $\|\cdot\|_p$ and
the $L^{p}$-norm of functions by $\|\cdot\|_{p,\Omega}$.
\section{Modeling of inpainting problem}\label{sec:Model}
Assume we want to reproduce a rectangular picture $\Omega'\subset\R^d$
by a tensor product spline. This can easily be done by choosing a
tensor product grid with multiple knots on the boundary
$\partial\Omega'$ such that $\Omega'$ is the domain of
definition. This method guarantees that the B-spline basis is stable.
\begin{theorem}
For any domain $R=\otimes_{j=1}^d[a_j,b_j]$, $d\in\N$ and any spline order
$\n=(n_1,\dots, n_d)\in\N^{d}$, the B-spline basis
$\{B_{\fa}^{\n}(\fx)\}_{\fa\in\ii_{\m,\n}} $ with respect to knots
defined as in \eqref{eq:knots} and \eqref{eq:boundary_knots} is stable,
that is, there exist constants $c,C>0$ such that
\[
c\|\mf\|_p\le \| \mB \mf \|_{p,\Omega}\le C \|\mf\|_p, \qquad \mf \in
\R^{I_\Omega},
\]
The constants are only depending on $\n$ and $d$.
\end{theorem}
Due to the tensor product structure and the multiple knots on the
boundary, this result can be proven analogously to the classical
univariate results for intervals, see, e.g., \cite{deB68}.
For further information about stability of tensor product B-spline
bases see for example \cite{MoessnerReif}. In addition to stability,
multiple knots on the boundary of $R$ help us to avoid artifacts at
the boundary since they result in an interpolation at the boundary.

For arbitrary $\Omega'\subset\R^d$, the inpainting problem considered here
can now be described in the following way:
Given two bounded domains $\Omega$ and $\Omega'$ such that
\begin{enumerate}
 \item $\Omega\subset\Omega'\subset\R^d$,
 \item there exists $R=\otimes_{j=1}^d[a_j,b_j]$ such that
 		$\Omega\subset R\subset\Omega'$ and
 		$\partial\Omega\cap\partial R=\emptyset$,
\end{enumerate}
reconstruct a function or picture $g$ on $\Omega$ that
is known only on $\Omega^{\ast}:=\Omega'\setminus\Omega$ by
using tensor product splines defined over $R$. The situation is
illustrated in Figure~\ref{fig:problem}.
\begin{figure}[ht]
\centering
\includegraphics[width=4cm]{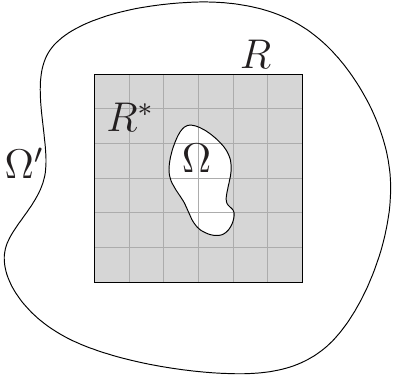}
\caption{Illustration of domains.}\label{fig:problem}
\end{figure}
We require that the spline fulfills a side
condition on $R^{\ast}:=R\setminus \Omega\subset\Omega^{\ast}$ and
minimizes a functional $F$ over some neighborhood $U(\Omega)$.
The general variational inpainting model using splines can now
be formulated as follows.
\begin{SIM}\label{SIM}
Let $U_{\lh}^{\n}(\Omega)$ be
some neighborhood of $\Omega$ only depending on $\lh$ and $\n$ with
$ \Omega\subseteq U_{\lh}^{\n}(\Omega)\subseteq R$.
The {\emph spline inpainting model} is given by:
Determine $s\in\TPSs(T,R)$ by
\begin{eqnarray*}
  \textrm{minimize}& 
  	\quad F(s,\nabla s,\dots)(\fx)
  						  & \textrm{for }\fx\in{U_{\lh}^{\n}(\Omega)}\\
  \textrm{subject to}&
  	\quad s(\fx)=g(\fx)&  \textrm{for }\fx\in R^\ast.
\end{eqnarray*}
\end{SIM}
\begin{remark}
Since hard constraints can be problematic, especially in the presence
of noise, the
minimization problem is often relaxed to
\begin{equation}\label{eq:Regularisation1}
\textrm{minimize}  \quad
F(s,\nabla s,\dots)|_{U_{\lh}^{\n}(\Omega)}+ 
\frac{\varepsilon}{2}\,
\| s-g\|_{2,R^{\ast}}^2\qquad \textrm{for some }\varepsilon>0.
\end{equation}
This relaxed formulation will be discussed briefly in
Section \ref{sec:Ex}. 
\end{remark}
For the explicit modeling or the discretization of the
Spline Inpainting model~\ref{SIM}, the concrete functional
in the minimization is, of course, crucial.
As already mentioned in Section~\ref{sec:Pre_Not}, we will follow the
most frequently used approach and focus
on TV inpainting \cite{CS02} and, therefore, on the ROF-functional
\cite{ROF92}. It seems worthwhile, however, to mention that the spline
approach also works with other functionals.
\begin{SIM}\label{TVSIM}
Let $U_{\lh}^{\n}(\Omega)$ be
some neighborhood of $\Omega$ only depending on $\lh$ and $\n$ with
$ \Omega\subseteq U_{\lh}^{\n}(\Omega)\subseteq R$.
The TV spline inpainting model is given by:
Determine $s\in\TPSs(T,R)$ by
\begin{eqnarray}\label{eq:min_cond}
  \textrm{minimize}& 
  	\quad \int_{U_{\lh}^{\n}(\Omega)}|\nabla s(\fx)|\, d\fx&\\
  	\label{eq:ide_cond}
  \textrm{subject to}&
  	\quad s(\fx)=g(\fx)&  \textrm{for }\fx\in R^\ast,
\end{eqnarray}
where $|\nabla s(\fx)| :=
\sqrt{\sum_{i=1}^d\big(\partial_i s(\fx)\big)^2}$.
\end{SIM}
In the next two subsections we show that the continuous
Spline Inpainting Models~\ref{SIM}-\ref{TVSIM} applied to discrete problems
(e.g., digital images) is already a discrete model, except for the
discretization of the integral. \\
Clearly, there are different possible interpretations of a discrete
image. In our situation, where we evaluate functions at points in $R$,
those interpretations influence the value of the image that we assume
to be at that points. Since it does not change the method or
modeling, we interpret, for the sake of simplicity,
the discrete values as
piecewise constant functions with constant values over rectangles.
Given an image $g$ consisting of $\mu_1 \times \cdots \times \mu_d$
pixels, those pixel rectangles are of the form
$P_{\fb}:=\otimes_{j=1}^d[a_j+(\beta_j-1)  p_j,a_j+\beta_j p_j)$
for $p_j:=(b_j-a_j)/\mu_j$ and $1\le \beta_j\le\mu_j$.
The value over $P_{\fb}$ is chosen corresponding to
the value at the {\emph center}
$c(P_{\fb}):=\big(a_j+(\beta_j-\frac{1}{2}) p_j\big)_{j=1}^d$;
the $j$-th coordinate will be denoted by
$c(P_{\fb})_j$. 
\subsection{Side condition}\label{sec:sidecondition}
In the reproduction of digital images on $R^\ast$, 
the side condition \eqref{eq:ide_cond} corresponds to the interpolation
of the pixel values in $R^\ast$.
In the space of tensor product splines
an interpolation problem is solvable if the number of interpolation
points is not greater than the degrees of freedom $\#(\ii_R)$ and if
there is a relationship between the interpolation sites and the knots,
known as the \emph{Schoenberg--Whitney condition}.
Since we also need some degrees of freedom for the minimization
problem over $U(\Omega)$, this forces us to choose a spline space of
sufficiently large dimension.
\begin{definition}
Given a tensor product grid $T$ over $R$ and a set of discrete points
$\Xi :=\{\fxi_{\fb}\in R|\,\fb\in\Z^d;\#(\fb)\le \#(\ii_R)\}$, let
$\lambda_{\fb}$
be the point evaluation functionals $\lambda_{\fb}(f):=f(\fxi_{\fb})$,
$f : \Xi \to \R$. A spline
$s\in \TPSs(T,R)$ is the {\em spline interpolant of a function
$f$ at $\Xi$}
if
\begin{equation*}
	\lambda_{\fb}(s)
	=
	\sum_{\fa\in \ii_R}f_{\fa} B_{\fa}^{\n}(\fxi_{\fb})
	=
	\lambda_{\fb}(f),\qquad
	\fxi_{\fb}\in\Xi.
\end{equation*}
\end{definition}
For functions defined over $R$ and grids T satisfying
\eqref{eq:knots} and \eqref{eq:boundary_knots} this interpolation
problem is uniquely solvable if the Schoenberg-Whitney
condition (see, e.g., \cite{HH13}) is satisfied. Since $R$ is a rectangle, this
can be guaranteed by choosing the tensor product \emph{Greville abscissae}
as interpolation sites, that is, by setting
\begin{equation}\label{eq:greville} 
	\Xi_{Gr}
	:=
	\bigg\{
	\fxi_{\fg}\in\R^d|\xi_{\fg,j}
	:=
	\frac{1}{n_j-1}\sum_{i=1}^{n_j-1}\tau_{j,\gamma_j+i},\,\fg\in \ii_R
	\bigg\}.
\end{equation}
It should be noted that the position of the Greville abscissae depends
on the position of the knots of $T$, hence, the
interpolation points depend on $T$. Therefore, the
tensor product knot grid must be chosen according to the interpolation
points. In case of digital images, where each rectangle $P_{\fb}$
represents
a pixel of the image, we have the following result for the reproduction
of all known image pixel values.
\begin{theorem}\label{la:sidecond}
Given an image $g$ of size $\otimes_{j=1}^{d}\mu_j$, an
area $R^\ast:=\cup_{\fb\in \iJ_{R^\ast}}P_{\fb}$,
$
\iJ_{R^\ast}
:=
\{\fb\in\Z^d:1\le \beta_j\le \mu_j,\,1\le j\le d\text{ and }
c(P_{\fb})\in R\setminus\Omega\}
$
of known pixel values 
and an order $\n$ of a spline space,
there exist a knot grid $T$ and a set of interpolation points $\Xi^\ast$
such that
	\[
	\{c(P_{\fb})\}_{\fb\in \iJ_{R^\ast}}\subset\Xi^{\ast}
	\qquad \text{and}\qquad
	\lambda_{\fg}(s)
	=
	\lambda_{\fg}(g),\qquad
	\fxi_{\fg}\in\Xi^{\ast},
	\]
for $s\in\TPSs(T,R)$.
\end{theorem}
\begin{proof}
  Due to the tensor product structure, it suffices to consider a
  single coordinate direction $j$ for the
construction of the grid and set of interpolation points.

In doing so, we have to distinguish between even and odd values of
$n_j$.
We set $m_j=\mu_j-1$ for $n_j$ odd and $m_j=\mu_j$ otherwise, and choose
\[
\begin{cases}
\tau_{j,n_j+k_j}:=a_j+k_j\cdot p_j
				& n_j \text{ odd},\\
\tau_{j,n_j+k_j}:=a_j+(k_j-\frac{1}{2})\cdot p_j
				& n_j \text{ even}
\end{cases}
\]
for $1\le k_j\le m_j$ and $p_j:=\frac{b_j-a_j}{\mu—j}$. The boundary
knots are determined with proper multiplicity according to
\eqref{eq:boundary_knots}.
Recalling \eqref{eq:greville}, the $j$-th coordinates of the
associated Greville abscissae for odd $n_j$ are given by
\begin{equation*}
\xi_{\fg,j}=
\begin{cases}
a_j
&
\text{for }\gamma_j =1,
\\
a_j + \frac{1}{n_j-1}\big(\sum_{\ell=1}^{\gamma_j-1}\ell\big)\,p_j
&
\text{for }2\le \gamma_j\le n_j-1,
\\
a_j + \frac{1}{n_j-1}\big(\sum_{\ell=\gamma_j+1-n_j}^{\gamma_j-1}\ell\big)\,p_j
&
\text{for }n_j\le \gamma_j\le m_j+1,
\\
b_j - \frac{1}{n_j-1}\big(\sum_{\ell=1}^{m_j+n_j-\gamma_j}\ell\big)\,p_j\phantom{p.}
&
\text{for }m_j+2\le \gamma_j\le m_j+n_j-1,
\\
b_j
&
\text{for }\gamma_j =m_j+n_j.
\end{cases}
\end{equation*}
Replacing the terms $\ell$ of the sums by $\ell-1/2$, results in the Greville
abscissae for even $n_j$, respectively. Therefore, we have
\begin{equation*}
  \xi_{\fg,j}=
  \begin{cases}
    a_j + \frac{(\gamma_j-1)\gamma_j}{2(n_j-1)}\,p_j
    &
    \text{for }1\le \gamma_j\le n_j-1,
    \\
    a_j + \big(\gamma_j-\frac{n_j}{2}\big)\,p_j
    &
    \text{for }n_j\le \gamma_j\le m_j+1,
    \\
    b_j - \frac{(m_j+n_j-\gamma_j+1)(m_j+n_j-\gamma_j)}{2(n_j-1)}\,p_j
    &
    \text{for }m_j+2\le \gamma_j\le m_j+n_j.
  \end{cases}
\end{equation*}
for odd $n_j$ and
\begin{equation*}
\xi_{\fg,j}=
\begin{cases}
a_j
&
\text{for }\gamma_j =1,
\\
a_j + \frac{(\gamma_j-1)^2}{2(n_j-1)}\,p_j
&
\text{for }2\le \gamma_j\le n_j-1,
\\
a_j + \big(\gamma_j-\frac{n_j+1}{2}\big)\,p_j
&
\text{for }n_j\le \gamma_j\le m_j+1,
\\
b_j - \frac{(m_j+n_j-\gamma_j)^2}{2(n_j-1)}\,p_j
&
\text{for }m_j+2\le \gamma_j\le m_j+n_j-1,
\\
b_j
&
\text{for }\gamma =m_j+n_j.
\end{cases}
\end{equation*}
for even $n_j$, respectively.
Now $\xi_{\fg,j}$ coincides, for $n_j\le\gamma_j\le m_j+1$,
with $c(P_{\fb})_j$ for some
$1\le\beta_j\le \mu_j$; more precisely, with centers $c(P_{\fb})_j$
such that
\[
\beta_j\in\mathring{\ii}:=
	\begin{cases}
	\{\frac{n_j+1}{2},\dots,\mu_j-\frac{n_j-1}{2}\}
		& \text{for odd } n_j\\
	\{\frac{n_j}{2},\dots,\mu_j+1-\frac{n_j}{2}\}
		&\text{for even } n_j
	\end{cases}.
\]
The Greville abscissae of \emph{boundary B-splines}, that is, of
B-splines with some knots at $a_j$ or $b_j$, are not necessarily
placed at a center, but they can be replaced
by the center that is closest to them without loss of the Schoenberg-Whitney
condition $B_{\gamma_j}^{n_j}(\xi_j)>0$, due to property
\eqref{eq:positivB}.

We now claim:
\emph{For every $\beta_j\in\{1,\dots,\mu_j\}\setminus \mathring{\ii}$,
there exists $\gamma_j$ such that 
\[
-\frac{1}{2}p_j<\xi_{\fg,j}-c(P_{\beta})_j\le\frac{1}{2}p_j.
\]}
There are four cases to be considered: small and large $\beta_j$ (or
$\gamma_j$) and even or odd order. We prove the claim for small
$\beta_j$ or $\gamma_j$
and odd order, the other cases can be verified in the same way.
For $\gamma_j\in\{2,\dots,n_j-1\}$ and odd $n_j$ we have
\[
\xi_{\fg,j}-c(P_{\beta})_j =
 \Big(\frac{(\gamma_j-1)\gamma_j}{2(n_j-1)}-\beta_j+\frac{1}{2}\Big)
 	\cdot p_j.
\]
Therefore, we need to show that for every
$\beta_j\in\{1,\dots,(n_j-1)/2\}$ there exists a
$\gamma_j\in\{2,\dots,n_j-1\}$ such that
\[
\beta_j-1 < \frac{(\gamma_j-1)\gamma_j}{2(n_j-1)}\le \beta_j.
\]
Set $f(\gamma_j):=\frac{(\gamma_j-1)\gamma_j}{2(n_j-1)}$. We have
$0<f(2)\le 1$ and $\frac{n_j-1}{2}-1<f(n_j-1)\le\frac{n_j-1}{2}$.
As a consequence, $0<f(\gamma_j)\le \frac{n_j-1}{2}$ for all
$\gamma_j\in\{2,\dots,n_j-1\}$.
Clearly, $f(\gamma_j)$ is monotonically increasing and $\#\gamma_j=n_j-2$
distributed over $(n_j-1)/2$ intervals. Further,
$f(\gamma_j)-f(\gamma_j-1)=\frac{\gamma_j-1}{n_j-1}\le 1$ and, thus, the
claim is valid.

Having this property at hand, we replace those $\xi_{\fg,j}$ by
$c(P_{\beta})_j$ that are closest to the center, to guarantee that all
centers are contained in the set of interpolation points.
For $\beta\in \{1,\dots,\mu_j\}\setminus \mathring{\ii}$
set
\[
\tilde{\ii}:=\Big\{\gamma_j\in\{1,\dots,m_j+n_j\}:
\xi_{\fg,j}=\min_{\tilde{\gamma}_j}
	\{\argmin_{\xi_{\tilde{\fg},j}}|\xi_{\tilde{\fg},j}-c(P_{\beta})_j|\}\Big\}
\]
and
\[
\tilde{\xi}_{\fg,j}:=\argmin_{c(P_{\beta})_{j}}
	|\xi_{\tilde{\fg},j}-c(P_{\beta})_j|
\quad\text{for }\gamma_j\in\tilde{\ii}.
\]
The set
\begin{equation*}
\Xi_j:=\big\{\xi_{\fg,j}:\,
		\gamma_j\in\{1,\dots,m_j+n_j\}\setminus \tilde{\ii}
	  \big\}
	 \cup 
	  \big\{\tilde{\xi}_{\fg,j}:\,
		\gamma_j\in \tilde{\ii}
	  \big\}
\end{equation*}
gives us the $j$-th coordinate of $\Xi$, the set of possible interpolation
points. Finally, the set $\Xi^{\ast}$ of interpolation points is determined by restriction
of ${\Xi}$ to $R^{\ast}$,
that is
$\fxi_{\fg}\in\Xi^{\ast}:={\Xi}\cap R^{\ast}$.
By construction, $c(P_{\fb})\in\Xi^{\ast}$ for $\fb\in \iJ_{R^{\ast}}$
and there exists a spline interpolant $s\in\TPSs(T,R)$ for $g$ at
$\Xi^{\ast}$.
\end{proof}
In the sequel, $\Xi$ and $\Xi^{\ast}$ always will refer to the sets of
interpolation sites determined in the way described in the proof of
Theorem \ref{la:sidecond}.
For applications we also want to describe the
side conditions in matrix notation. For the coefficients
$\mathbf{f}=\big(f_{\fa}\big)_{\fa\in \ii_R}\in\R^{\#(\ii_R)}$, we get
\begin{equation}\label{eq:interpolant}
	\mathbf{s}_{\Xi^{\ast}}
	=
	\mathbf{B}_{\Xi^{\ast}} \mathbf{f}
	=
	\mathbf{g}_{\Xi^{\ast}}
\end{equation}
with
\[
\mathbf{g}_{\Xi^{\ast}}
=
\big(g(\xi_{\fb})\big)_{\fb\in{\Xi^{\ast}}}\in\R^{\#({\Xi^{\ast}})}, \qquad
\mathbf{B}_{\Xi^{\ast}}
=
\bigg[
B_{\fa}^{\n}(\xi_{\fb}):
	\begin{array}{l}
	\xi_{\fb}\in{\Xi^{\ast}}\\
	\fa \in \ii_R
	\end{array}			 
\bigg]
\in\R^{\#({\Xi^{\ast}})\times\#(\ii_R)}.
\]

\subsection{Minimization}
Depending on $\Omega$, there are still some degrees of freedom left for
minimization; their number is given by $\# (\Xi\setminus\Xi^{\ast})$.
We define $U_{T}^{\n}(\Omega)$ 
as the union of all
supports of B-splines corresponding to
$\xi_{\fa}\in\Xi\setminus\Xi^{\ast}$, that is
\begin{equation}\label{eq:U}
U_{T}^{\n}(\Omega)
:=
\bigcup_{\fa\in \ii_R\setminus\ii_{\Xi^{\ast}}} S^{\n}_{\fa}.
\end{equation}
The index set of all cells in $U_{T}^{\n}(\Omega)$ is denoted by
$\ik_\Omega
:=
\{\fk\in \ik|\,Z_{\fk}\subset U_{T}^{\n}(\Omega)\}$. 

\begin{Pro}\label{la:min}
Given $T$ and $\Omega\subset R$. For $U_{T}^{\n}(\Omega)$ defined
as in \eqref{eq:U}, the minimization problem \eqref{eq:min_cond}
reduces to
\[
\min_{\mf}
\quad
\sum_{\fk\in\ik_\Omega}
\int_{Z_{\fk}}|\big(\mB^{j}(\fx)\mf\big)_{j=1}^d|\, d\fx,
\]
where $\mf=\big(f_{\fa}\big)_{\fa\in \ii_\Omega}$, $f_{\fa}\in\R$,
and $[\mB^{j}(\fx)]^T
:=\big(\partial_jB_{\fa}^{\n}(\fx)\big)_{\fa \in \ii_R}
\in\R^{\#(\ii_R)}$.
\end{Pro}

\begin{remark}
It is worthwhile to point out that the derivatives
$\partial_jB_{\fa}^{\n}(\fx)$
are just weigh\-ted differences of B-splines of lower order,
the explicit formula being given in Section \ref{sec:Pre_Not}. This
fact is of crucial importance for the efficient implementation of the
method as it leads to sparse difference matrices.
\end{remark}

\begin{proof}[Proof of Proposition \ref{la:min}]
The result follows by applying \eqref{eq:derivativeS}
and \eqref{eq:U}.
The first one directly gives us
$\nabla s(x) = \big(\mB^{j}(\fx)\mf\big)_{j=1}^d$ with unknown
coefficients $\mf$. The area of integration $U_{T}^{\n}(\Omega)$
consists of a union of grid cells and, thus, the integral can be
split up into those grid cells $Z_{\fk}$ for $\fk\in\ik_\Omega$.
\end{proof}

Typically, for digital images the integration and the
derivatives need to be discretized, as already mentioned in
Section~\ref{sec:Intro}. 
In our case, we only need to discretize the integral and the
way the problem is modeled helps with that, too: the
discretization of the integrals can be done simply and efficiently by
using a Gaussian quadrature formula for the grid cells
$Z_{\fk}$ for $\fk\in\ik_\Omega$.
The tensor product structure allows us to use the univariate
Gauss-Legendre quadrature formula in each coordinate direction. Let
$\Theta\subset U_{T}^{\n}(\Omega)$ be the set of the nodes of the
Gauss-Legendre quadrature and $w_{\ftheta}$ be the corresponding
weights.
We get
\begin{equation*}
\sum_{\fk\in\ik_\Omega}
\int_{Z_{\fk}}|\big(\mB^{j}(\fx)\mf\big)_{j=1}^d|\, d\fx
\approx
\sum_{\ftheta}w_{\ftheta}\,
	\big\|\big(
	\mB^{j}(\ftheta)\mf
	\big)_{j=1}^d\big\|_2,
\end{equation*}
where $\|\cdot\|_2$ denotes the Euclidean norm.

\begin{dSIM}\label{dTVSIM}

Using Proposition \ref{la:min} and \ref{la:sidecond} (or \eqref{eq:interpolant})
the discrete form of TV Spline Inpainting Model \ref{TVSIM} is given by:
Determine $s\in\TPSs(T,R)$ by
\begin{eqnarray*}
  \textrm{minimize}_{f_{\fa}}& 
  	\quad \sum_{\ftheta\in\Theta}w_{\ftheta}\,\|\mB_{\ftheta}\mf\|_2&\\
  \textrm{subject to}&
  	\quad \mathbf{B}_{\Xi^{\ast}} \mathbf{f}=\mathbf{g}_{\Xi^{\ast}}
\end{eqnarray*}
where $\mB_{\ftheta}
:=\big(\mB^{j}(\ftheta)\big)_{j=1}^d\in\R^{d\times\#(\ii_R)}$.
\end{dSIM}
This approach is also suitable for other types of Spline
Inpainting Models \ref{SIM}, as long as the functionals
depend on the function and its derivatives and it can be applied as
soon as the explicit functional is given.
\section{Numerical implementation and experiments}\label{sec:Ex}

Define the convex operators $F:\R^{d\times\#(\Theta)}\rightarrow\R$ and $G:\R^{\#(\ii_\Omega)}\rightarrow\R$ with 
\begin{align*}
F(x_1,\ldots,x_{\#(\Theta)})=\sum\limits_{k=1}^{\#(\Theta)}\|x_k\|_2,
&& G(\mf) = \begin{cases}
0 & \mB_{\Xi^{\ast}}\mf=\mathbf{g}_{\Xi^{\ast}}, \\
\infty & \text{otherwise},
\end{cases}
&& 
\end{align*}
and the linear operator $K:\R^{\#(\ii_\Omega)}\rightarrow\R^{2\times\#(\Theta)}$ with $K(\mf)=(w_{\ftheta}\mB_{\ftheta}\mf)_{\ftheta\in\Theta}$. Using this, the Discrete Spline Inpainting Model 1 can be reformulated as an optimization problem in the following way:

\begin{OP}\label{OP}
Our Spline Inpainting Model~\ref{SIM} is equivalent to the unconstrained problem
\begin{eqnarray*}
\textrm{minimize} &\quad & F(K\mf)+G(\mf).
\end{eqnarray*}
\end{OP}

This formulation is suitable for the alternating method of multipliers (ADMM) \cite{CP11}. Therefore, the $\lambda$-proximity operators $\mathop{prox}_{\lambda F^*}$, $\mathop{prox}_{\lambda G}$ need to be calculated. Note that the convex conjugate function $F^*:\R^{d\times\#(\Theta)}\rightarrow\R$ is given by
\begin{align*}
F^*(y_1,\ldots,y_{\#(\Theta)})=\begin{cases}0 & \|y_k\|_2\leq1,\ \forall k\leq\#(\Theta) \\ \infty & \text{otherwise}\end{cases}.
\end{align*}
Both, $F^*$ and $G$ are indicator functions and, thus, the proximity operators are equivalent to the projections:
\begin{align*}
\mathop{prox}_{\lambda F^*}(y_1,\ldots,y_{\#(\Theta)})&=\left(\frac{y_1}{\max(1,\|y_1\|_2)},\ldots,\frac{y_{\#(\Theta)}}{\max(1,\|y_{\#\Theta}\|_2)}\right), \\
\mathop{prox}_{\lambda G}(\mf)&=\mf-\mB_{\Xi^{\ast}}^{+}(\mB_{\Xi^{\ast}}\mf-\mathbf{g}_{\Xi^{\ast}}),
\end{align*}
where $\mB_{\Xi^{\ast}}^{+}$ is the pseudoinverse of
$\mB_{\Xi^{\ast}}$. We remark that both operators are
independent of $\lambda$. In our experiments, we use the MATLAB
implementation of ADMM due to
Gabriel Peyre \cite{PHP} and apply the tests to
several
cartoon-like images and natural images in different sizes; some are
shown in Figure \ref{fig_testData}. 

\begin{figure}
\centering
\includegraphics[width=0.49\textwidth]{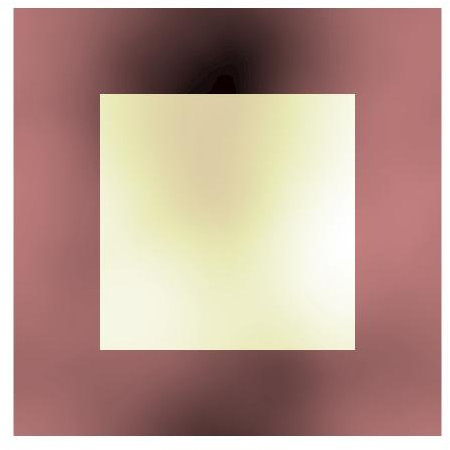}
\includegraphics[width=0.49\textwidth]{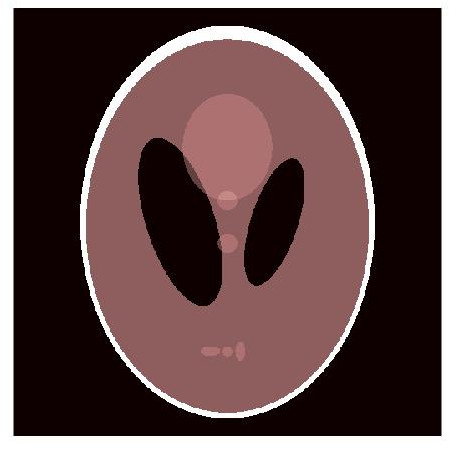}\\
\includegraphics[width=0.49\textwidth]{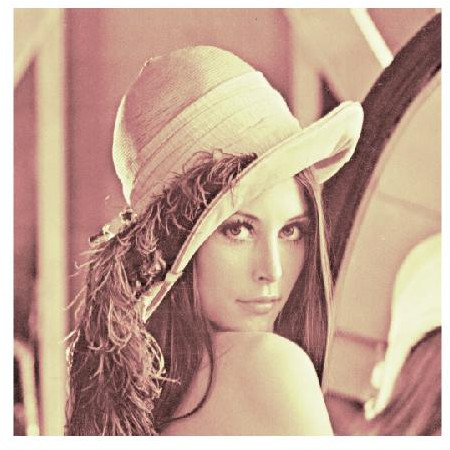}
\includegraphics[width=0.49\textwidth]{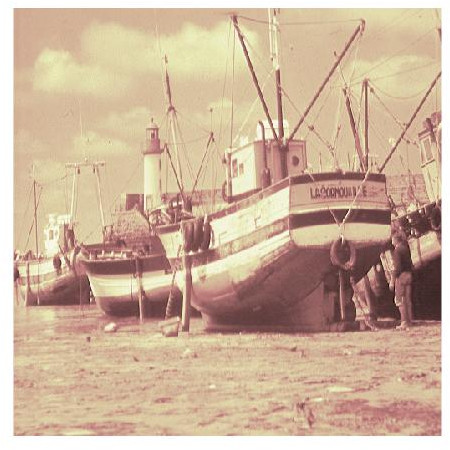}
\caption{Exemplary test data: Cartoon-like images (top) and natural images.}
\label{fig_testData}
\end{figure}

We consider two different scenarios for the inpainting region $\Omega$. In one scenario, the image is damaged by one or several "scratches" of variable width; in the second case we consider randomly missing pixels similar to "salt-and-pepper" noise. Figure \ref{fig_noise} gives an example for both types (in comparison to Figure \ref{fig_testData}
the contrast is changed to emphasize the inpainting area).

\begin{figure}
\centering
\includegraphics[width=0.49\textwidth]{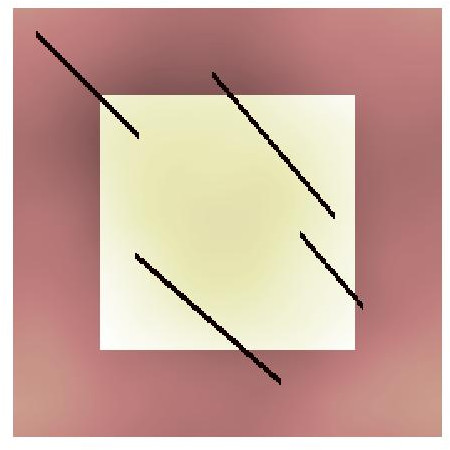}
\includegraphics[width=0.49\textwidth]{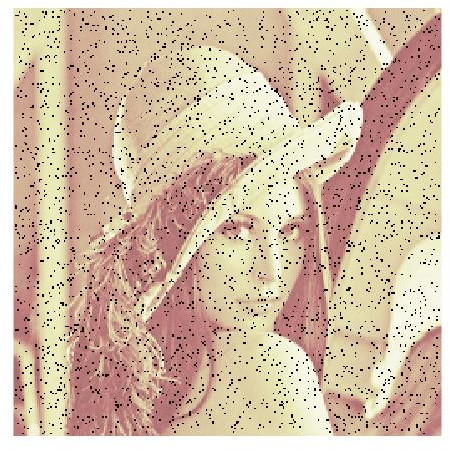}
\caption{Considered inpainting area types: Scratches (left) and randomly missing pixels (right).}
\label{fig_noise}
\end{figure}

In the next subsection we discuss the optimal spline order for the
inpainting problem as well as reasonable strategies to guess a
starting value for ADMM. We compare our results with the standard TV
inpainting using, again, an implementation by Gabriel Peyre
\cite{PHP}. Afterwards, we demonstrate the benefits of our method by
two examples: Text removal and salt-and-pepper denoising.

\subsection{Numerical evaluation}

\subsubsection{Starting guess}

Both implementations, spline inpainting using ADMM and standard TV
inpainting, are iterative methods. Thus, they rely on a suitable
initial value to return a good solution after a reasonable number of
iterations. We tried several strategies among which two stood out and will
be detailed in what follows. In a first strategy we choose random
uniformly distributed values in $[0,255]$ (for standard grayscale
images). Note that standard TV inpainting directly iterates on the
pixel values of the
inpainting area while our spline approach works on the spline coefficients
as its variables. Thus, although both algorithms use uniformly distributed
values, the starting guess strategies differ slightly. For our second
strategy, we calculate the mean value $\omega_{\text{mean}}$ of the
image outside the inpainting area. In contrast to that, standard TV
sets the pixels inside
the inpainting area to $\omega_{\text{mean}}$. Our spline
approach uses the starting values $\mf=\mathop{prox}_{\lambda
  G}(\omega_{\text{mean}}{\mathbf 1})$, where $\mathbf{1}$ is a
vector of ones. Since splines form a partition of unity, the
coefficient vector $\omega_{\text{mean}}{\mathbf 1}$ generates a
constant image with value $\omega_{\text{mean}}$. Using
$\mathop{prox}_{\lambda G}$, this vector is projected onto the
interpolation space. 

Figure \ref{fig_starting} illustrates the mean signal-to-noise ratio (SNR) for both starting strategies using $100$ 
experiments on several images ($128\times128$ pixels) and inpainting areas with $100$ iterations of ADMM. We use two types of inpainting areas: on the one hand $3\%$ of the pixels are randomly set to zero (left), on the other hand $3$ scratches with a $4$ pixel width are used (right). The mean SNR is calculated for cartoon-like images (top) and natural images (bottom) separately.

We see that the mean value starting guess performs much better on
natural images and in case of cartoon like images the SNR values are
at least of the same order as for a random starting guess. Only for
standard TV on scratch inpainting areas the reconstruction quality of
the random method outperforms the mean value. Thus, we suggest to use
the mean value starting guess for spline inpainting, while for
standard TV a random guess on cartoon-like images and the mean value
on natural images should be chosen. This will also be the strategies
used in the following experiments.

Figure \ref{fig_start_example} shows the reconstruction of the "Lena"
image shown in Figure \ref{fig_noise} with randomly missing pixels. We
use splines of order $2$; first with random starting guess and a
second time with mean value. The obtained SNR values are $20.00$ and
$26.47$, respectively.

\begin{figure}
\centering
\includegraphics[width=0.49\textwidth]{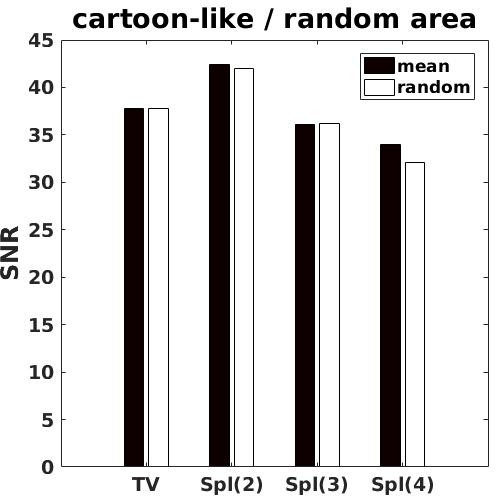}
\includegraphics[width=0.49\textwidth]{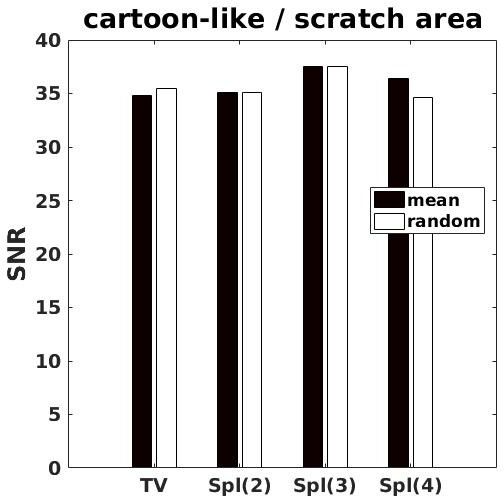}\\
\includegraphics[width=0.49\textwidth]{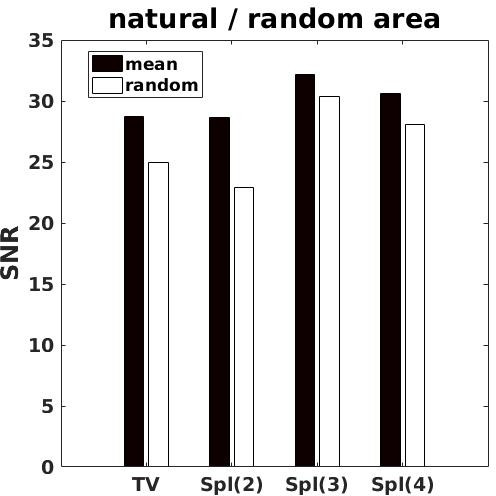}
\includegraphics[width=0.49\textwidth]{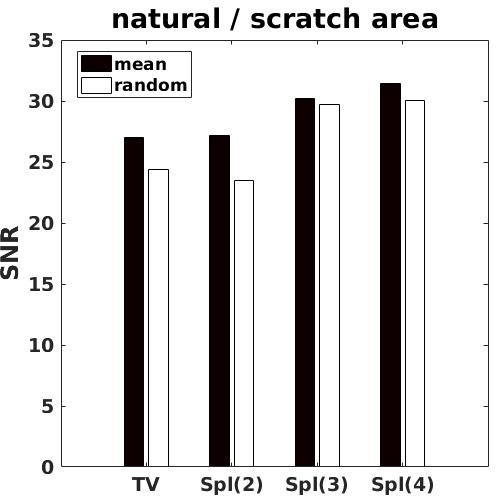}
\caption{Obtained SNR values using random starting guess (white) or mean value (black). Left: random inpainting area (3 \% of the pixels); right: scratches (3 scratches, 4 pixels width); top: cartoon-like images; bottom: natural images.}
\label{fig_starting}
\end{figure}

\begin{figure}
\centering
\includegraphics[width=0.49\textwidth]{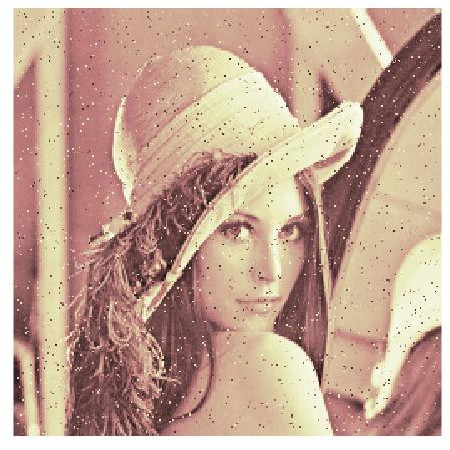}
\includegraphics[width=0.49\textwidth]{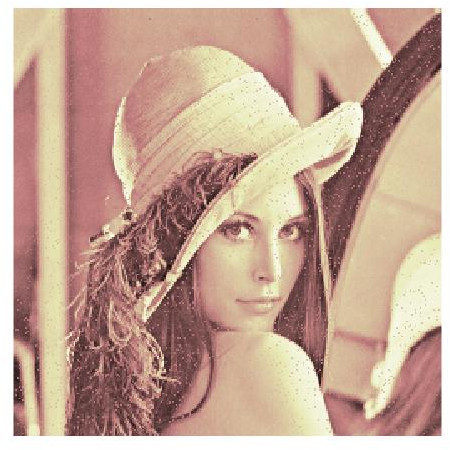}
\caption{Reconstruction of Lena (see Figure \ref{fig_noise}) using splines of order $2$ with random starting guess (left) and mean value strategy (right).}
\label{fig_start_example}
\end{figure}

\subsubsection{Spline order}
Figure \ref{fig_starting} already makes clear that a higher spline order not necessarily yields a better reconstruction. Indeed, the optimal order seems to depend on the image type as well as on the inpainting area. Higher order splines are preferable for natural images, which usually have a more complex structure, as well as for inpainting regions that produce larger gaps, e.g., scratches. For more simple structures and smaller inpainting regions lower order splines return comparable or even better results.
This will be analyzed in more details in the next experiment.

Again, we calculate the mean SNR over $100$ runs using $128\times128$ pixel images. This time, we plot the SNR against the percentage of unknown pixels (random inpainting area) or scratch width (scratch inpainting area) for both image types. Figure \ref{fig_order} shows the obtained SNR for different spline orders and standard TV as comparison. For both algorithms, we use the optimal starting guess just discussed. 

Obviously, natural images are better reconstructed using splines of order $3$ or higher. Especially, when the inpainting area is a connected scratch, higher orders can improve the result. When the image geometry becomes more simple, as e.g., for cartoon-like images, splines of order $2-3$ perform best. For this reason, we recommend to use splines of order $2$ or $3$ for cartoon-like images and order $3$ to $4$ for natural images, depending on the inpainting region. Note that standard TV inpainting performs comparable to spline inpainting with order $2$ for natural images. For cartoon-like images and randomly missing pixels it even performs better than all splines of order $\geq3$, but is outperformed by spline order $2$.

Figure \ref{fig_order_example} gives an example how the reconstruction quality can increase if a higher order spline is used on natural images. The scratched image was reconstructed using splines of order $2$ and $3$ obtaining a SNR of $25.73$ and $30.67$. As a comparison, the SNR in case of standard TV is $25.72$.

\begin{figure}
\centering
\includegraphics[width=0.49\textwidth]{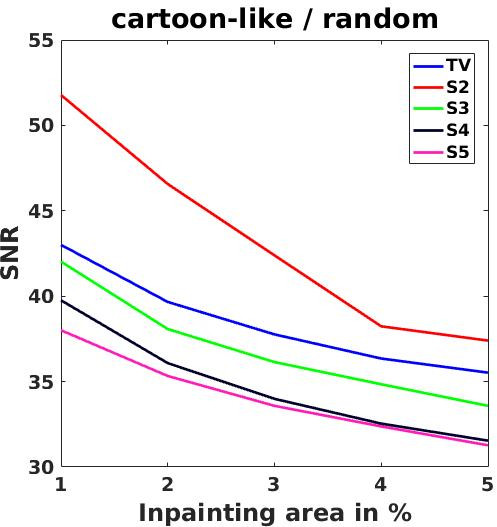}
\includegraphics[width=0.49\textwidth]{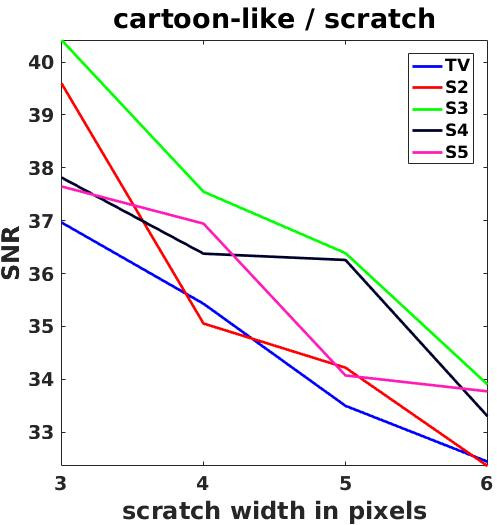}\\
\includegraphics[width=0.49\textwidth]{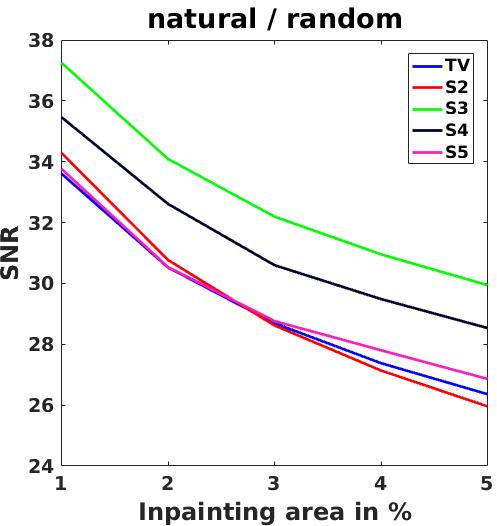}
\includegraphics[width=0.49\textwidth]{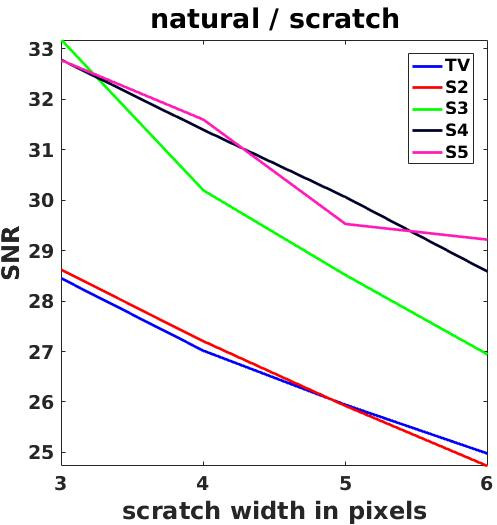}
\caption{Obtained SNR values for different spline orders and standard TV. Left: random inpainting area; right: scratches; top: cartoon-like images; bottom: natural images.}
\label{fig_order}
\end{figure}

\begin{figure}
\centering
\includegraphics[width=0.49\textwidth]{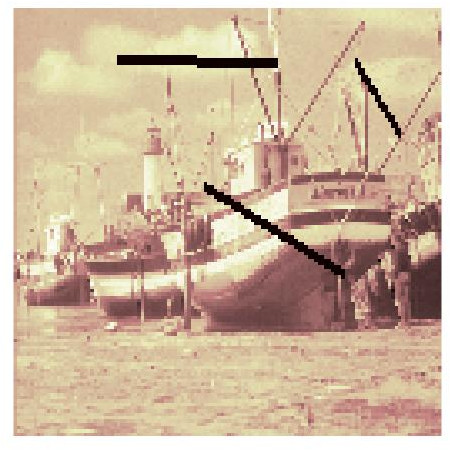}
\includegraphics[width=0.49\textwidth]{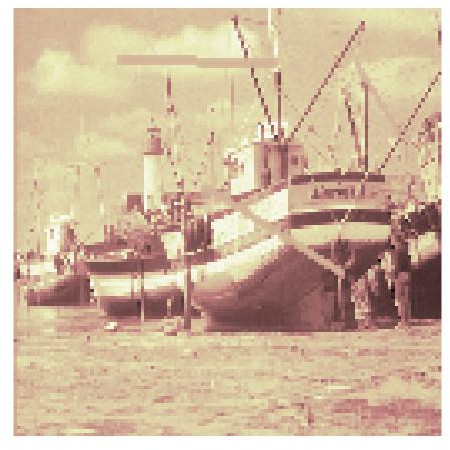}\\
\includegraphics[width=0.49\textwidth]{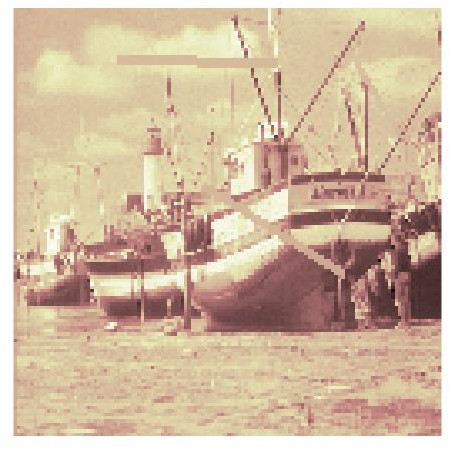}
\includegraphics[width=0.49\textwidth]{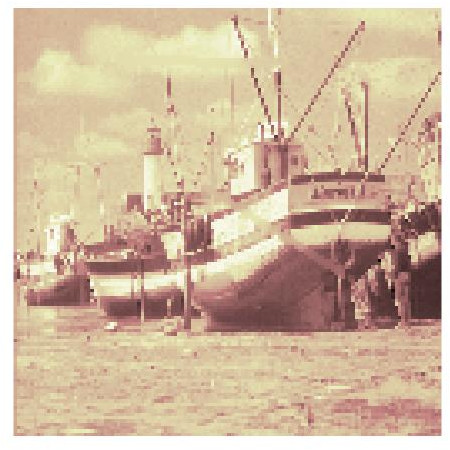}
\caption{Top, left: Original image with scratches; top, right: reconstruction using standard TV inpainting (SNR $25.72$); bottom, left: spline inpainting order 2 (SNR $25.73$); bottom, right: spline inpainting order 3 (SNR $30.67$).}
\label{fig_order_example}
\end{figure}

\subsection{Application 1: text removal}

An application of inpainting methods is the removal of unwanted objects in images, such as text. In Figure \ref{fig_text} we illustrate how an image might be covered by an advert or other text. Here, we use the "Lena" image of size $256\times256$ pixels. We compare the  TV reconstruction with our spline approach of order $3$, $4$ and $5$. Since we use a natural image and the text inpainting area creates large gaps in the data, the reconstruction quality strongly profits from a higher spline order. The reconstruction has an SNR of only $23.2$ for spline order $3$, but increases to $28.45$ for order $4$ and even $31.25$ for order $5$. Standard TV inpainting results in an SNR of $23.58$.

\begin{figure}
\centering
\includegraphics[width=0.49\textwidth]{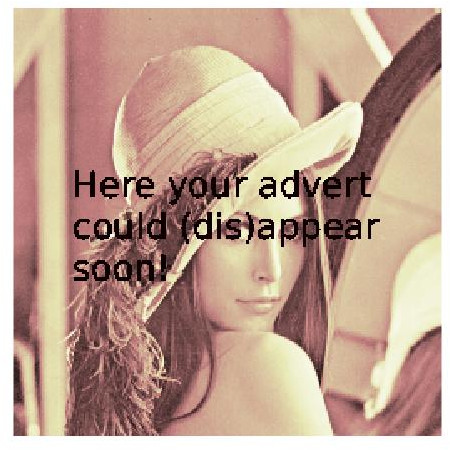}
\includegraphics[width=0.49\textwidth]{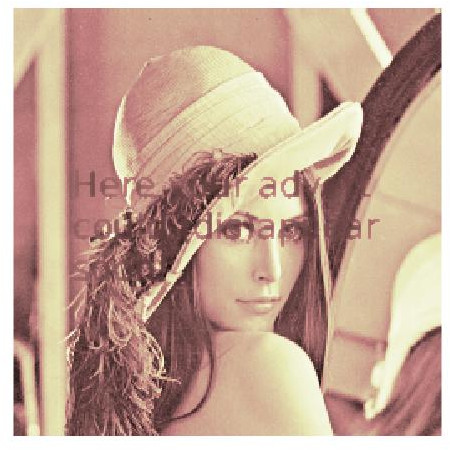}\\
\includegraphics[width=0.49\textwidth]{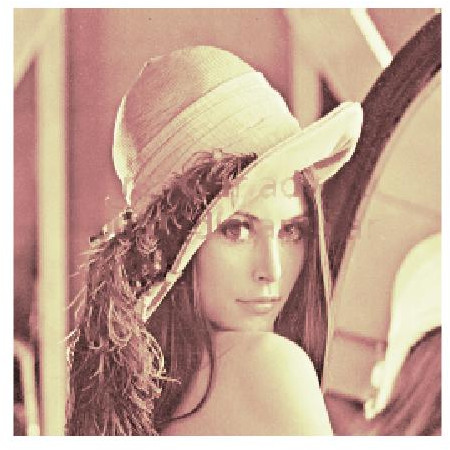}
\includegraphics[width=0.49\textwidth]{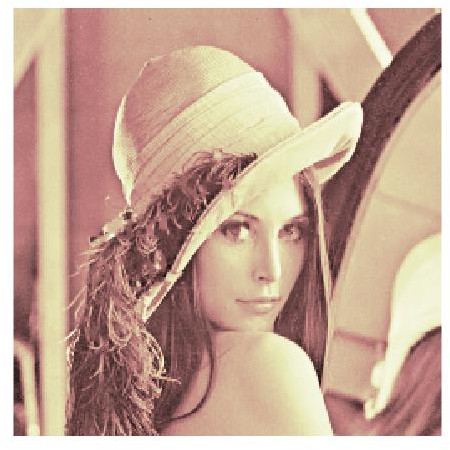}
\caption{Removing text from Lena: text image (top, left), spline inpainting using order $3$ (top, right), order $4$ (bottom, left) and $5$ (bottom, right). Obtained SNR values: $23.2$, $28.45$ and $31.25$.}
\label{fig_text}
\end{figure}

\subsection{Application 2: Salt-and-pepper noise}

As a last example, we consider an image that is corrupted by salt-and-pepper noise, i.e., some pixels are randomly set to the minimal or maximal value. Since this noise deletes all information about the original pixel value, we can model the reconstruction as inpainting problem where the inpainting region is given by all pixels with minimal or maximal value. This coincides with our random inpainting area model. 

So far, we discussed noiseless data outside the inpainting domain. Therefore, the operator $G$ was chosen such that only interpolating solutions were allowed. However, when dealing with noisy images it is not unlikely that the image is corrupted by several types of noise. Hence, we now use the relaxed model \eqref{eq:Regularisation1}. We can still use the ADMM, but now with the operator $G(\mf)=\frac{\varepsilon}{2}\|\mB_{\Xi^{\ast}}\mf-\mathbf{g}_{\Xi^{\ast}}\|_2^2$ and its proximity operator
\begin{align*}
\mathop{prox}_{\lambda G}(\mf)= (\mathop{\mathbf I}+\lambda\varepsilon\mB_{\Xi^{\ast}}^T\mB_{\Xi^{\ast}})^{+}
(\mf+\lambda\varepsilon\mB_{\Xi^{\ast}}^T\mathbf{g}_{\Xi^{\ast}}),
\end{align*}
where $\mathop{\mathbf I}$ is the identity matrix. Note that we cannot use the relaxed model \eqref{eq:Regularisation1} to denoise the image far away from the inpainting domain since the TV minimization is only performed in a surrounding of the unknown pixels. Hence, for a good reconstruction we need to choose a higher spline order such that the spline support enlarges. Moreover, a large constant $\varepsilon$ should be chosen, otherwise the solution of the relaxed model \eqref{eq:Regularisation1} is nearly constant. A simultaneous denoising may be obtained by expanding the TV term in model \eqref{eq:Regularisation1} to the complete image. However, this will drastically increase the number of Gauss-Legendre points used for the quadrature formula and, thus, increase the computational complexity.

Figure \ref{fig_denoise} shows the result of spline inpainting with order $4$ using the relaxed model on an image that was corrupted by Gaussian and salt-and-pepper noise. The parameter $\varepsilon$ is chosen such that the best SNR is obtained. This is illustrated in Figure \ref{fig_denoise_parameter} where the SNR for different choices of $\varepsilon$ is shown. We notice that for small $\varepsilon$ the solution is nearly constant while a too large parameter can reproduce more of the noise in the reconstruction. Figure \ref{fig_denoise} shows the reconstruction for $\varepsilon=50$ which leads to an SNR of $16.7$. As a comparison, the SNR of the noised image is $4.04$ in case of Gaussian and salt-and-pepper noise, and $15.69$ if only Gaussian noise is considered.

\begin{figure}
\centering
\includegraphics[width=0.49\textwidth]{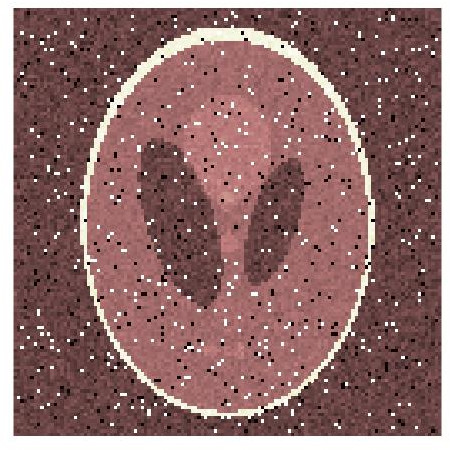}
\includegraphics[width=0.49\textwidth]{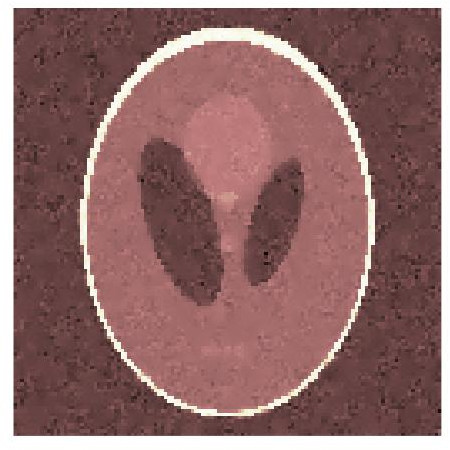}
\caption{Cartoon-like image corrupted by Gaussian and salt-and-pepper noise (left) and its reconstruction using spline inpainting with order $4$ (right); SNR $16.7$.}
\label{fig_denoise}
\end{figure}

\begin{figure}
\centering
\includegraphics[width=0.49\textwidth]{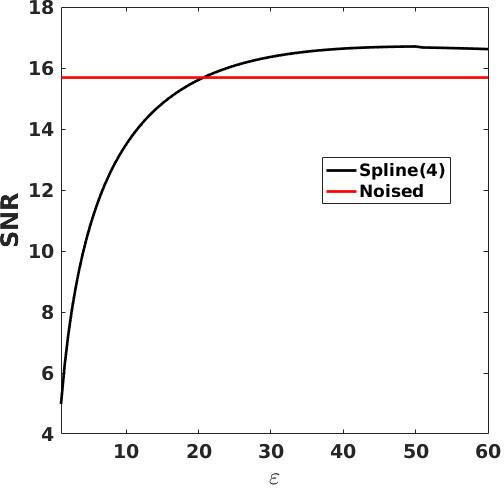}
\caption{SNR value obtained by spline inpainting with order $4$ and different parameters $\varepsilon$. The SNR of the Gaussian noised image is plotted as comparison.}
\label{fig_denoise_parameter}
\end{figure}

\section{Conclusion}

We presented a new approach to model the discrete inpainting problem using TV regularization and splines. The spline order can be chosen adapted to the underlying image and inpainting area. The advantages of this method were demonstrated in numerical experiments. Especially, when the images are complex and/or the inpainting domains are large, the reconstruction quality can highly profit from the new concepts by choosing a higher spline order. But also for cartoon-like images and low spline orders
the method returns reasonable results. As a slight disadvantage, the new methods requires precalculation of the spline basis, Greville abscissae and Gauss-Legendre points what increases the runtime of the algorithm.

Although the method increases the reconstruction quality, it is still based on TV minimization and, thus, cannot overcome the known issues of TV inpainting completely. Therefore, considering other functionals is future work. Moreover, we want to extend the method on more general domains. Depending on the image, it might not be necessary to interpolate every pixel which could
be exploited by the use of non-uniform knots (and interpolation points) and/or the combination with hierarchical splines.

\bibliographystyle{siamplain}
\bibliography{Lit_inp.bib}
\end{document}